\documentclass[letterpaper,11pt]{article}

\usepackage{amsthm}
\usepackage{amsmath}

\usepackage{amsfonts}
\usepackage{amscd,amsthm,amsfonts,amsopn,amssymb,verbatim}

\newtheorem{thm}{Theorem}[section]

\newtheorem{lemma}[thm]{Lemma}
\newtheorem{prop}[thm]{Proposition}
\newtheorem{cor}[thm]{Corollary}

\newtheorem{question}[thm]{Question}

\newcommand{\beq}[1]{\begin{equation}\label{#1}}
\newcommand{\enq}[0]{\end{equation}}

\newcommand{\bn}[0]{\bigskip\noindent}
\newcommand{\mn}[0]{\medskip\noindent}
\newcommand{\nin}[0]{\noindent}

\newcommand{\sub}[0]{\subseteq}
\newcommand{\sm}[0]{\setminus}
\renewcommand{\dots}[0]{,\ldots,}

\newcommand{\A}[0]{{\cal A}}
\newcommand{\B}[0]{{\cal B}}

\newcommand{\f}[0]{{\cal F}}

\newcommand{\h}[0]{{\cal H}}

\newcommand{\K}[0]{{\cal K}}
\newcommand{\m}[0]{{\cal M}}

\newcommand{\Q}[0]{{\cal Q}}
\newcommand{\R}[0]{{\cal R}}

\newcommand{\T}[0]{{\cal T}}

\newcommand{\rrr}[0]{{R}}
\newcommand{\rr}[0]{{_\rrr}}

\newcommand{\ra}[0]{\rightarrow}

\newcommand{\TT}[0]{{\bf T}}
\newcommand{\UU}[0]{{\bf U}}

\newcommand{\XX}[0]{{\bf X}}

\newcommand{\aaa}[0]{\mbox{{\sf a}}}
\newcommand{\bbb}[0]{\mbox{{\sf b}}}
\newcommand{\ccc}[0]{\mbox{{\sf c}}}
\newcommand{\ddd}[0]{\mbox{{\sf d}}}

\newcommand{\ttt}[0]{t}

\newcommand{\0}[0]{\emptyset}

\renewcommand{\qed}[0]{\begin{flushright} \rule{2mm}{3mm} \end{flushright}}

\newcommand{\C}[2]{{{#1}\choose{{#2}}}}
\newcommand{\Cc}[0]{\tbinom}
\newcommand{\ga}[0]{\alpha }
\newcommand{\gb}[0]{\beta }

\newcommand{\gd}[0]{\delta }
\newcommand{\gD}[0]{\Delta }
\newcommand{\gG}[0]{\Gamma }

\newcommand{\gl}[0]{\lambda }
\newcommand{\gL}[0]{\Lambda}

\newcommand{\gO}[0]{\Omega}

\newcommand{\gS}[0]{\Sigma}
\newcommand{\gz}[0]{\zeta}
\newcommand{\eps}[0]{\varepsilon }
\newcommand{\vt}[0]{\vartheta}

\newcommand{\vp}[0]{\varphi}

\newcommand{\his}[0]{\h_i^\star}
\newcommand{\ais}[0]{A_i^\star}

\newcommand{\sugg}[1]{}

\newcommand{\comments}[1]{}
\usepackage[usenames]{color}
\begin{document}

\renewcommand{\thefootnote}{\fnsymbol{footnote}}

\title{
On Erd\H{o}s-Ko-Rado for random hypergraphs II\footnotemark}
\author{A. Hamm and J. Kahn}
\date{}
\footnotetext{ $^*$Supported by NSF grant DMS1201337}

\maketitle

\begin{abstract}
Denote by $\h_k(n,p)$ the random $k$-graph
in which each $k$-subset of $\{1\dots n\}$ is present
with probability $p$, independent of other choices.
More or less answering a question of Balogh, Bohman and Mubayi,
we show:  there is a fixed $\eps>0$ such that
if $n=2k+1$ and $p> 1-\eps$, then w.h.p. (that is, with probability tending to 1
as $k\ra \infty$), $\h_k(n,p)$ has the ``Erd\H{o}s-Ko-Rado property."
We also mention a similar random version of Sperner's Theorem.

\end{abstract}

\section{Introduction}\label{Intro}

One of the most interesting
combinatorial trends of the last couple decades
has been the investigation of ``sparse random" versions of
some of the
classical theorems of the subject---that is, of
the extent to which
such results hold in a random setting.
This issue has been the subject some
spectacular successes, particularly those related to the
theorems of Ramsey \cite{Ramsey}, Tur\'an \cite{Turan} and Szemer\'edi \cite{Sz};
see \cite{FR86,BSS,RR, KLR}
for origins and, e.g., \cite{Conlon-Gowers, Schacht,DK,BMS,ST}
(or the survey \cite{Rodl-Schacht})
for a few of the more recent developments.

Here we are interested in the analogous question
for the Erd\H{o}s-Ko-Rado
Theorem \cite{EKR},
another cornerstone of extremal combinatorics.
This natural problem has already been considered by
Balogh, Bohman and Mubayi \cite{BBM}, and we first
quickly recall a few of the notions from that paper.

In what follows $k$ and $n$ are always positive integers with
$n>2k$.  As usual we write $[n]$ for $\{1\dots n\}$ and
$\C{V}{k}$ for the collection of $k$-subsets of a set $V$.
A $k$-{\em graph} (or $k$-{\em uniform hypergraph}) on $V$
is a multisubset, say $\h$, of $\C{V}{k}$.
Members of $V$ and $\h$ are called {\em vertices} and {\em edges} respectively.
We use $\h_x$ for the set of edges containing $x$ ($\in V$),
called the {\em star} of $x$ in $\h$
or the {\em principal} subhypergraph generated by $x$.
For the present discussion we take $V=[n]$ and write $\K$ for $\C{V}{k}$.

A collection of sets is {\em intersecting}, or a {\em clique},
if no two
of its members are disjoint.
The Erd\H{o}s-Ko-Rado Theorem says that for any $n$ and $k$
as above, the maximum size of an intersecting $k$-graph on $V$
is $\C{n-1}{k-1}$ and, moreover, this bound is achieved only
by the stars.
%

Following \cite{BBM} we say
$\h$ satisfies {\em (strong) EKR} if
every largest clique of $\h$ is a star;
thus the EKR Theorem says $\C{V}{k}$ satisfies EKR.
(We also say, again as in \cite{BBM}, that
$\h$ satisfies
{\em weak EKR} if {\em some} largest clique
is a star, but this slightly weaker notion
will not concern us here.)

For the rest of this introduction we use $\h=\h_k(n,p)$ for the random $k$-graph
on $V$ in which members of $\C{V}{k}$ are present independently,
each with probability $p$.
As suggested above, we are interested in understanding when
EKR holds for $\h$; a little more formally:

\begin{question}
For what $p_0=p_0(n,k)$
is it true that $\h$ satisfies EKR w.h.p. provided $p> p_0$?
\end{question}
\nin
(As usual ``w.h.p." ({\em with high probability}) means with
probability tending to one as $n\ra\infty$.)

\medskip
The nature of the problem may be said to change around
$k=\sqrt{n}$, since for $k$ smaller than this, two random
$k$-sets are typically disjoint, while the opposite is
true for larger $k$.
Heuristically we may say that the problem becomes more interesting/challenging
as $k$ grows and the potential
violations of EKR proliferate (though increasing $k$ does narrow the range of $p$ for which
we {\em expect} EKR to hold).

In this paper we are interested in what happens when $k$ is as
large as possible.  The next assertion is our main result.

\begin{thm}\label{MT}
There is a fixed $\eps>0$ such that if $n=2k+1$ and $p>1-\eps$, then
$\h $ satisfies EKR w.h.p.
\end{thm}

\nin
This was prompted by Question 1.4 of \cite{BBM}, {\em viz.}
\begin{question}
Is it true that for $k\in (n/2-\sqrt{n},n/2)$ and $p=.99$,
EKR (or weak EKR) holds w.h.p. for $\h$?
\end{question}

\nin
Note that for $n,k$ as in Theorem~\ref{MT},
EKR is unlikely unless
$p$ is large
(so ``sparse random" is something of a misnomer),
since a simple calculation shows that
for $p$ less than
about $3/4$ stars are unlikely even to be maxi{\em mal} cliques.
(This is, of course, reminiscent of the Hilton-Milner Theorem \cite{H-M}, which says that
(for any $k$ and $n>2k$) the largest nontrivial cliques in $\C{[n]}{k}$ are those
of the form $\{A\}\cup\{B\in\C{[n]}{k}:x\in B, B\cap A\neq \0\}$
(with $A\in \C{[n]}{k}$ and $x\in [n]\sm A$).)
We expect that, for $k,n$ as in Theorem~\ref{MT},
this is in fact the main hurdle---that is,
EKR becomes likely as soon as stars are likely to be maximal---but we
are far from proving such a statement.
On the other hand, as will appear below, the main difficulties in proving the theorem
involve cliques that are far from stars.

We haven't thought very hard about whether the
$\eps$ in Theorem~\ref{MT} could be pushed
to .01, since this seems somewhat beside the point
(and since it seems not wildly unethical to
regard ``$.99$" as really meaning ``$1-\eps$ for some fixed $\eps>0$").
We assume our methods could be adapted to give Theorem~\ref{MT}
for smaller $k$,
but confine ourselves to the present statement.  This is
partly for simplicity, but also because we don't believe the theorem gives a very
satisfactory answer in other cases; e.g. even for $n=2k+2$ we expect EKR
to hold for $p$ down to about $1/k$.

\medskip
The original paper of Balogh {\em et al.} dealt mostly with $k< n^{1/2-\eps}$
(for a fixed $\eps>0$).
In a companion paper \cite{HK2} we precisely settle the
question for $k$ up to about $\sqrt{(1/4)n\log n}$ and suggest a possible general answer.

\medskip
The rest of this paper is organized as follows. Section~\ref{Preliminaries} sets notation and fills in some
mostly standard background, and Section~\ref{Setting} reduces Theorem~\ref{MT}
to a related, slightly fussier statement.
The most interesting part of the argument, given in Section~\ref{Proof},
proves the latter statement using, in addition to standard large deviation considerations,
asymptotic-enumerative ideas inspired especially by work of A.A. Sapozhenko \cite{Sap}.
A final short section mentions a counterpart of Theorem~\ref{MT} for
Sperner's Theorem
that follows easily from the method developed in Section~\ref{Proof}.

\section{Preliminaries}\label{Preliminaries}

{\em Usage}

Set $M=\C{2k}{k-1}$ and $N=\C{2k}{k}$.
Unless specified otherwise, we use $\K$ for $\C{[n]}{k}$.
As usual, $2^S$ is the power set of $S$
and, for a hypergraph $\h$,
$d_\h(x)$ is the degree of $x\in V$ in $\h$ (i.e. $|\{A\in \h:x\in A\}|$)
and $ \gD_{\h} $ is the maximum of these degrees.

For graphs, $xy$ is an edge joining vertices $x$ and $y$; $N(x)$ is, as usual, the
neighborhood of $x$ (and $N(X)=\cup_{x\in X}N(x)$);
$\nabla(X,Y)$ is the set of edges joining the disjoint vertex sets $X,Y$;
and $d_W(x)=|N(x)\cap W|$ (for $W\sub V$).

We use $B(m,\ga)$ for a random variable
with the binomial distribution ${\rm Bin}(m,\ga)$ and
$\log $ for $\ln$.
We assume throughout that $n=2k+1$ is large enough to support our arguments.

\bn
{\em Large deviations}

We use
Chernoff's inequality in the following form, which
may be found, for example, in
\cite[Theorem 2.1]{JLR}.
\begin{thm}\label{Chern}
For $\xi =B(m,q)$, $\mu=mq$ and any $\gl\geq 0$,
\begin{eqnarray*}
\mathbb{P}(\xi >\mu+\gl)& < &\exp [- \tfrac{\gl^2}{2(\mu+\gl/3)}],\\
\mathbb{P}(\xi < \mu-\gl) &<& \exp [- \tfrac{\gl^2}{2\mu}].
\end{eqnarray*}
\end{thm}

\nin
We will also need the following improvement
for larger deviations, for which see e.g. \cite[Theorem A.1.12]{AS}.
\begin{thm}\label{uppertail}
For $\xi =B(m,q)$ and any $K$,
\begin{eqnarray*}
\mathbb{P}(\xi > Kmq)< \exp[-Kmq\log (K/e)].
\end{eqnarray*}
\end{thm}

\nin
(Of course this is only meaningful if $K>e$.)

\bn
{\em Isoperimetry and degree}

For $A\sub \C{[2k]}{k}$ let
$
\gd(A)= (|\partial_u A|-|A|)/|A|,
$
where $\partial_u A =\{y\in \C{[2k]}{k+1}:\exists x\in A, y\supset x\}$
(the {\em upper shadow} of $A$).
We will use the following consequence of
the Kruskal-Katona Theorem
(\cite{Kruskal}, \cite{Katona} or e.g. \cite{Bollobas}).

\begin{prop}\label{KKProp}
For $A\sub \C{[2k]}{k}$ with $|A|\leq N/2$,
\beq{delta1}
\gd(A) \geq \tfrac{\log 2}{k}~\log_2\left(\tfrac{N}{2|A|}\right).
\enq
\end{prop}
\nin
(Recall $N=\C{2k}{k}$, and notice that $N/2=\C{2k-1}{k}$.
The $\log 2$ in \eqref{delta1} can probably be replaced by 1, but cannot be replaced by $k/(k-1)$.)

\mn
{\em Proof.}
We use Lov\'asz' version \cite[Problem 13.31]{Lovasz}
of Kruskal-Katona, which in the present situation says that if $|A| = \C{x}{k}$
($:= (x)_k/k!$ for any $x\in \mathbb{R}$), then
$|\partial_u(A)|\geq \C{x}{k-1}$.
(This is ordinarily stated for the {\em lower} shadow, which is equivalent here
since our universe is of size $2k$.)

Let $|A| = \C{2k-t}{k}$, noting that $|A|\leq N/2$ implies $t\geq 1$,
and $\psi = k^{-1}\log 2$.
Then $\frac{N}{2|A|} = \frac{(2k)_k}{2(2k-t)_k}$
and, from Kruskal-Katona (Lov\'asz),
\[
\gd(A)\geq \Cc{2k-t}{k-1}/\Cc{2k-t}{k}-1= \frac{t-1}{k-t+1}.
\]
Thus \eqref{delta1} will follow from
\[
f(t):=\frac{t-1}{k-t+1}-\psi\log_2\left[\frac{(2k)_k}{2(2k-t)_k}\right]\geq 0
~~~~~\mbox{for $t\geq 1$,}
\]
so (since $f(1)=0$) from
$f'(t)\geq 0~$.
But, recalling the value of $\psi$, we have
\[
f'(t) = \frac{k}{(k-t+1)^2}-\frac{1}{k}\sum_{i=0}^{k-1}\frac{1}{2k-t-i}
\geq \frac{k}{(k-t+1)^2}-\frac{1}{k-t+1} \geq 0.
\]\qed

\bigskip
The following result of P. Frankl \cite{Frankl} will also be helpful
in getting things started.
(We give the result for general $k$, $n$ and $i$, again writing $\K$ for
$\Cc{[n]}{k}$, but will
only use it with $n=2k+1$ and $i=3$.)
Given $k$ and $n>2k$, set, for each $i\in \{3\dots k+1\}$,
%
%
\[
\f_i= \{A\in \K: 1\in A, A\cap \{2\dots i\}\neq\0\}\cup
\{A\in \K: A\supseteq \{2\dots i\}\}.
\]

\begin{thm}[\cite{Frankl}]\label{TFrankl}
For any k, n and i as above, if $\f\sub \K$ is a clique with
$|\f|>|\f_i|$, then $\gD_{\f}> \gD_{\f_i}$.
\end{thm}

\bn
{\em Graphs}

Two special graph-theoretic notions will be relevant in
what follows.  First, for a bigraph $\gS$
with bipartition $\gG_1\cup \gG_2$, the {\em closure} of
$X\sub \gG_i$ is $[X]=\{x:N(x)\sub N(X)\}$
(and $X $ is {\em closed} if it is equal to its closure).
Second, for a (general) graph
$\gS$
and positive integer $j$,
$W\sub V(\gS)$
is {\it $j$-linked} if for all $u,v\in W$ there
are $u=u_0,u_1,\dots ,u_l=v$
with $u_i\in W$ ($\forall i$) and $\rho(u_{i-1},u_i)\leq j$
for $i\in [l]$, where $\rho$ is graph-theoretic
distance.
We will eventually need the following observation from
\cite{Sap}.
\begin{prop}\label{Plink}
Let $\gS$ be a graph
and suppose $A$ and $T$ are subsets of $V(\gS)$ with
$T\sub N(A)$, $A\sub N(T)$ and $A$ $j$-linked.
Then $T$ is $(j+2)$-linked.
\end{prop}
\nin
{\em Proof.}  Given $u,v\in T$, choose
$x,y\in A$ with $x\sim u$, $y\sim v$, and then
$x=x_0\dots x_\ell=y$ with $x_i\in A$ and
$\rho(x_{i-1},x_i)\leq j$ ($i\in [\ell]$).
If we now let $u_0=u$, $u_\ell=v$ and
$x_i\sim u_i\in T$ for $i\in [\ell-1]$,
then
$\rho(u_{i-1},u_i)\leq 1 +\rho(x_{i-1},x_i)+1\leq j+2$
(for $i\in [\ell]$).  The proposition follows.\qed

\medskip
We also find some use for the following standard bound.
\begin{prop}\label{Knuth}
In any graph with all degrees at most $d$,
the number of trees of size $u$ rooted at
some specified vertex is at most $(ed)^{u-1}$.
\end{prop}
\nin
{\em Proof.}
This follows easily from
the fact (see e.g. \cite[p.396, Ex.11]{Kn}) that
the infinite $d$-branching rooted tree contains precisely
$\frac{1}{(d-1)u+1}{ du \choose u}\leq (ed)^{u-1}$
rooted subtrees of size $u$.\qed

\bn
{\em Etc.}

We make repeated use of the fact that for positive integers $a,b$
with $a\leq b/2$,
\beq{Cnt}
\sum_{i\leq a}\Cc{b}{i}\leq \exp[a\log (eb/a)].
\enq

\section{Setting up}\label{Setting}

\mn
In what follows,
$\h$ denotes a member of $\m$,
the collection of {\em nonprincipal}
maximal intersecting families
in $\C{[n]}{k}$.
We now set $\h_k(n,p)=\XX$, where $p=1-\eps$, with $\eps>0$ fixed
but small enough to support our arguments.
(We make no attempt to optimize.)

\medskip
The statement we are to prove is
\beq{toshow2}
\mbox{w.h.p. $\max_{\h\in \m}|\XX\cap \h|<\max_x|\XX\cap \K_x|$,}
\enq
but we will find it better to work with a variant, \eqref{toshow3} below.
This requires a little preparation.
(Though getting to \eqref{toshow3} does require a little effort, we would
stress that the main interest of the present
work is in the proof of
\eqref{toshow3}---really meaning the proof of \eqref{toshow4}---in Section~\ref{Proof}.)

\medskip
For $x\in [n]$ and $0\leq \ell\leq n-1$,
let $\gG^x_\ell $ denote the collection of $\ell$-subsets
of $[n]\sm \{x\}$.
Let $\gS^x$
be the usual bigraph on $\gG^x_k\cup\gG^x_{k+1}$
(that is, with
adjacency given by set containment), and write
$N^x$ for neighborhood in $\gS^x$.
For $A\sub \gG_k^x$ set $\gd_x(A) = (|N^x(A)|-|A|)/|A|$
(so
$N^x(A)$ is the upper shadow of $A$
in $2^{[n]\sm \{x\}}$ and our usage here follows that in
Proposition~\ref{KKProp}).

For
$\h\in \m$ (and $x\in [n]$),
let $A^x(\h)=\h\sm\K_x$, $J^x(\h)=\K_x\sm \h$ and
$G^x(\h)=N^x(A^x(\h))$;
thus $A^x(\h)$
and $G^x(\h)$
are subsets of $\gG^x_k$ and
$\gG^x_{k+1}$ respectively.
Note that
\beq{AJHF}
|\XX\cap A^x(\h)|- |\XX\cap J^x(\h)| =
|\XX\cap \h|- |\XX\cap \K_x|.
\enq
For the next two paragraphs we fix $\h\in \m$ and use
$A^x$ for $A^x(\h)$ and similarly for $J^x$ and $G^x$.

For $\B\sub 2^{[n]}$ set $\B^c=\{[n]\sm T:T\in \B\}$.
It is easy to see that
maximality
of $\h$ implies that $G^x =(J^x)^c$ (for any $x$).
It is not quite true that maximality also implies that the sets
$A^x$ are closed (in $\gS^x$), but they are mostly not far from being so,
as follows.

With $\T^x=\T^x(\h) = [A^x]\sm A^x$
(recall $[\cdot ]$ is closure) and $T\in \K$ not containing $x$, we have
$T\in \T^x$ iff $T$ meets all members of $\h_x$
(equivalently, $N^x(T)\sub G^x$) and
$[n]\sm (\{x\}\cup T)\in \h$.
In particular,
\beq{Tx1}
\mbox{the sets $\T^x$ are pairwise disjoint}
\enq
(since for $T\in \T^x$ and $y\not\in T\cup \{x\}$,
$T$ misses $[n]\sm (\{x\}\cup T)\in \h_y$, implying $T\not\in \T^y$).
Moreover, $\T=\T(\h):=\cup_x\T^x$ is relatively small:
\beq{Tsize}
|\T|< (1+1/k)(|\K|-2|\h|) = (1+1/k)[2(M-|\h|)+M/k].
\enq
To see this, let $\nabla(\h)$ be the set of disjoint pairs $(S,T)\in \K^2$ with $S\in \h$
and (therefore) $T\in \K\sm \h$.  (In other language $\nabla(\h)$ is the edge-boundary of $\h$
in the Kneser graph $K(n,k)$.)
Since each $T\in \T$ belongs to exactly one such pair
(and no member of $\K\sm \h$ belongs to more than $k+1$), we have
\[
(k+1)|\h|  = |\nabla(\h)| \leq |\T| + (k+1)(|\K|-|\h|-|\T|),
\]
which rearranges to \eqref{Tsize}.

\medskip
Let $\Q$ be the event that there are $\h\in \m$ and
$x\in [n]$ for which
$A^x(\h)$ is 2-linked (in $\gS^x$),
\beq{dxa'}
\gd_x([A^x(\h)])> 1/(3k),
\enq
and
$|\XX\cap \h|\geq |\XX\cap \K_x|$.
Our main point, the aforementioned variant of \eqref{toshow2}, is
\beq{toshow3}
\mathbb{P}(\Q)=o(1).
\enq

\medskip
Before proving this (in Section~\ref{Proof}), we
show that it implies \eqref{toshow2}, by showing that
failure of
\eqref{toshow2} implies $\Q$.
Supposing \eqref{toshow2} fails,
let $\h\in \m$ maximize
$|\XX\cap \h|$ and,
with $\h^* =\{T\in \K:|T\cap [3]|\geq 2\}$,
choose $x\in [n]$ as follows.
If $|\h| > |\h^*|$ then let
$x$ be some vertex with $d_{\h}(x)=\gD(\h)$;
otherwise,
choose $\gl=\gl(n)$ satisfying $1\ll \gl < \sqrt{n}$
and let $x$ be some vertex with
\beq{Tx}
\mbox{$|A^x| < (1+2/\gl)(k+1)|\h|/n~$ and $~ |\T^x|<\gl^{-1}|\T|$.}
\enq
(Existence is given by Markov's Inequality:  since $\sum |A^x| = (k+1)|\h|$, the number of $x$'s
violating at least one of the two conditions in \eqref{Tx} is at most
$((1+2/\gl)^{-1} +\gl/n)n < n$.)

Let $A=A^x(\h)$, $J=J^x(\h)$, $G=G^x(\h)$ and $\T^x=\T^x(\h)$.
By \eqref{AJHF} (and our assumption that
$|\XX\cap \h|\geq|\XX\cap \K_y| ~\forall y$)
we have
\[  
|\XX\cap A|\geq |\XX\cap J|  .
\]  

Suppose first that $A$ is 2-linked in $\gS^x$.
In this case we claim $(\h,x)$ itself satisfies
$\Q$, i.e.\ that
\eqref{dxa'} holds.
If $|\h| > |\h^*|$, then Theorem~\ref{TFrankl} gives
$\gD(\h)>\gD(\h^*) \sim 3M/4$,
whence (using \eqref{Tsize} and noting that here $|\K|-2|\h|= o(M)$),
\[
|[A]| =|A|+|\T^x|\leq |A|+|\T| <(1+o(1)) M/4;
\]
so \eqref{dxa'} is given by \eqref{delta1}.
If, on the other hand, $|\h|\leq |\h^*|$, then,
noting that $M-|\h^*|\sim M/(4k)$, we find that
\[
\gd_x([A]) > (1-o(1))/(2k)
\]
(so also \eqref{dxa'}) follows from
\beq{A}
|[A]| = |A|+|\T^x| < (1+o(1)) M/2
\enq
and
\begin{eqnarray}\label{B}
|G|-|[A]|& =& |J|-|A|-|\T^x| ~ =~ M-|\h| -|\T^x|\\
&\sim & M-|\h| ~\geq~M-|\h^*| ~>~ (1-o(1))M/(4k).
\end{eqnarray}
Here \eqref{A} follows from \eqref{Tx} (and \eqref{Tsize}).
For the ``$\sim$" in \eqref{B} note that,
since $M-|\h| =\gO(M/k)$,
\eqref{Tsize} and the second part of \eqref{Tx} give $|\T^x| = o (M-|\h|)$.

\medskip
Now suppose $A$ is not 2-linked, and let
$A_1\dots A_s$ be its 2-linked components
(defined in the obvious way),
$G_i=N^x(A_i)$, $J_i= (G_i)^c$
(so the $G_i$'s and $J_i$'s partition $G$ and $J$ respectively)
and $\h_i=(\K_x\sm J_i)\cup A_i$.

The $\h_i$'s are intersecting but not necessarily maximal, so for each $i$ we fix
some maximal intersecting $\his\supseteq \h_i$ and set $\ais=\his\sm\K_x$
($=A^x(\his)$ once we know $\his\in\m$).
Notice that
\beq{hi1}
\his\sm\h_i\sub [A_i]\sm A_i ~~(\sub \T^x)
\enq
(since  $\K_x\sm J_i$ consists precisely of those sets on $x$ that meet all sets in $A_i$,
we have
$
\his\sm \h_i\sub \{T\in \K\sm \K_x: T\cap S\neq \0 ~\forall S\in \K_x\sm J_i\} = [A_i])
$.
In particular \eqref{hi1} implies
\beq{Gix}
G^x(\his)=G_i,
\enq
\[
\mbox{$\ais$ is 2-linked}
\]
(immediate from \eqref{Gix} and the fact that $A_i$ is 2-linked) and
\beq{his}
\his\in\m
\enq
(that is, $\his$ is not principal).
For \eqref{his} note that $\his\neq\K_x$, since $\ais\neq \0$,
while $\his =\K_y$ ($y\neq x$) requires $G_i=\{S\in \gG_{k+1}:y\in S\}$,
implying that any $T\in \gG_k$ is 2-linked to $A_i$ and contradicting
the assumption that $A_i\neq A$.

\medskip
Suppose w.l.o.g. that $|A_1|=\max_i|A_i|$.
Then for $i\geq 2$ we have
\[
|[A_i]|\leq |A_i|+|\T^x|\leq |A|/2+|\T^x|< (1/4+o(1))M
\]
(where the second inequality is again given by \eqref{Tsize} if
$|\h| > |\h^*|$ and by \eqref{Tx} otherwise).
Thus, again using \eqref{delta1}, we have
$\gd_x([A_i])>(\log 2-o(1))/k$.
So we have $\Q$ (at $\his$, $x$) if
$|\XX\cap A_i|\geq |\XX\cap J_i|$ (recall
$\ais\supseteq A_i$ and $\K_x\sm \his =J_i$)
for some $i\geq 2$;
but if this is not the case then
(again using \eqref{AJHF})
\begin{eqnarray*}
|\XX\cap \h_1^\star|- |\XX\cap \K_x|
&=&|\XX\cap A_1^\star|-|\XX\cap J_1| \\
&\geq&
|\XX\cap A|-|\XX\cap J|
-\sum_{i\geq 2}(|\XX\cap A_i|-|\XX\cap J_i|)\\
&>& |\XX\cap A|-|\XX\cap J|
=|\XX\cap \h|- |\XX\cap \K_x|,
\end{eqnarray*}
contradicting the assumed maximality of $|\XX\cap \h|$.\qed

\section{Main point}\label{Proof}

Here we prove \eqref{toshow3}.
For the remainder of our discussion we work with a
fixed $x\in [n]$
and drop the super- and subscripts $x$ from our notation;
so to begin, we set
$\gS^x=\gS$ and $\gG^x_\ell=\gG_\ell$.
We will use
$G_A$ for the neighborhood of
$A\sub \gG_k$ in $\gS$
and
\[
\gd(A) = |G_A|/|A|-1  ~~(=\gd_x(A)) .
\]
We stress immediately that $A$ is now a general subset
of $\gG_k$, not necessarily $A^x(\h)$ for some $\h\in \m$.
(It will soon be a general \emph{closed} subset.)

\medskip
We extend $\XX$ to
$\gG_{k+1}$ by declaring that
$T\in \XX$ iff $[n]\sm T\in \XX$
(so here $T$ is a $(k+1)$-set off $x$ and $[n]\sm T$ is a $k$-set on $x$); we may then forget
about $J(\h)$ ($=J^x(\h)$) and
regard $\XX$ as a subset of $\gG_k\cup\gG_{k+1}$.
Note that ({\em cf.} \eqref{AJHF})
``$|\XX\cap \h|\geq |\XX\cap \K_x|$" in the definition of $\Q$ is then the same
as ``$|\XX\cap A|\geq |\XX\cap G_A|$" when $A=\h\sm\K_x$ and (thus) $G_A =J^x(\h)^c$,
and that this (trivially) implies $|\XX\cap [A]|\geq |\XX\cap G_A|$.

\medskip
For the proof of \eqref{toshow3} we will bound the
probability that $\Q$ occurs at our given $x$ with specified sizes of
$[A^x(\h)]$ and $G^x(\h)$ (so of $A$ and $G_A$ if we take $A=[A^x(\h)]$),
and then sum over possibilities for these sizes.
(Of course we need a bound $o(1/n)$ since we must eventually sum over $x$.)
Thus
we assume throughout that we have fixed $a,g$ with
\beq{dxa''}
\gd:=(g-a)/a > \max\{1/(3k),
(\log 2/k)\log_2(N/(2a))\}
\enq
(with the second term in the max again given by
Proposition~\ref{KKProp}),
and write $\A=\A(a,g)$ for the set of $A$'s satisfying
\beq{Ahyps}
\mbox{$A$ is closed and 2-linked,
$|A|=a$ and $|G_A|=g$}.
\enq
Notice that for $A\in \A$ we have
\begin{eqnarray}
|\nabla(G_A,\gG_k\sm A)| &=& (k+1)g-ka
\nonumber\\
&=& (k+1)(1+\gd)a -ka
=(1+(k+1)\gd )a.\label{after10}
\end{eqnarray}

Let $\Q(a,g)$ (= $\Q_x(a,g)$) be the event
that there is some $A\in \A(a,g)$ with
\beq{XGA}
|\XX\cap G_A|\leq |\XX\cap A|.
\enq
We show
\beq{toshow4}
\sum_{a,g}\mathbb{P}(\Q(a,g)) =o(1/n),
\enq
which, since the union of the $\Q(a,g)$'s is implied by occurrence of $\Q$ at $x$,
gives \eqref{toshow3}.

\medskip
The bound \eqref{toshow4} is (of course)
the heart of the matter, and the rest of our discussion is
devoted to its proof.
This turns out to be rather delicate, and a rough
indication of where we are headed may be helpful.
(The following description refers to the main case, namely $\gd\leq 1$,
considered below.)

For $A\in \A$ we have
\beq{EXGA}
\mathbb{E} |\XX\cap G_A|- \mathbb{E} |\XX\cap A| =\gd a p,
\enq
so can rule out \eqref{XGA} if we can say that the
quantities
$|\XX\cap G_A|$ and $|\XX\cap A|$
are close to their
expectations, where ``close" means somewhat small
relative to $\gd ap$ ($\approx \gd a$).
The problem (of course) is that
though each of these {\em individual} events is unlikely,
there are too many of them to allow a simple union bound.

Our remedy for this is to exploit similarities among the
$A$'s (and similarly $G_A$'s, but for this very rough
description we stick to $A$'s) to avoid paying repeatedly
for the same unlikely events.
To do this we specify each $A\in \A$ {\em via} several
``approximations," beginning with a set
$S_A$ for which $A\gD S_A$ is
{\em fairly} small, and then adding
and subtracting lesser pieces.
It will then follow that $|\XX\cap A|$ is close to its
expectation provided this is true of $|\XX\cap B|$ for each of
the relevant pieces $B$.

Thus
we will want to say that,
with $B$ ranging over some
to-be-specified collection of subsets of $\gG_k$, it is likely that
all $|\XX\cap B|$'s are close to their expectations.
Of course the probability that this fails for a particular $B$
grows with $|B|$ (since the benchmark $\gd ap$ does not change),
so we would like to arrange
that the larger $B$'s
are not too numerous.
For example, the above
$S_A$'s will necessarily be large
(of size roughly $a$), but there will be relatively few
of them, reflecting the fact that
a single $S$ will typically
be $S_A$ for many $A$'s.
We may think of $\A$
as consisting of a large number of variations on a
relatively small number of themes, though, as we will see,
controlling these themes and variations turns out to be not very straightforward.

As mentioned earlier, our approach here
has its roots in the beautiful ideas of A.A. Sapozhenko
\cite{Sap}, which were originally developed to deal with ``Dedekind's
Problem" and related questions in asymptotic enumeration.

\bn
{\em Proof of} \eqref{toshow4}.
As our fixed $x$ plays no further role in what follows, we will feel free to recycle
and use ``$x$" (along with $u,v,y,z$) to denote a general member of our ground set,
which we may now think of as $[2k]$.

\medskip
We divide the proof of \eqref{toshow4}
into two cases, large and small $\gd$, beginning with
the second, which is by far
the more interesting.
(Our treatment of this case can be adapted to work in general---actually with
most of the contortions below becoming unnecessary and/or vacuous---but this
seems pointless given how much simpler the proof is for large $\gd$.
It also seems worth stressing that, as mentioned earlier,
the real challenge is in dealing with quite small $\gd$
(and thus, according to \eqref{delta1}, with quite large $a$).)

\medskip
Assume then that $\gd\leq 1$ (say), and note that in this case \eqref{delta1}
gives
\beq{alarge}
a>(4/e)^k(4\sqrt{k})^{-1}=:a_0
\enq
(which is pretty far from the truth
but we have plenty of room here).

\mn
{\em Prospectus.}
Before we continue, some further pointers may be helpful.

This main part of our argument proceeds in two phases.
At the end of the first phase we will have associated with each
$A\in\A$ several sets (drawn from $\gG_k$, $\gG_{k+1}$ and $E(\gS)=\nabla(\gG_k,\gG_{k+1})$)
from which decent approximations of $A$ and $G_A$ can be built up in a useable way.

The output of this phase, summarized in the paragraphs following the proof of
Lemma~\ref{Ulemma}, is a collection, $\R$, of triples encapsulating the relevant information;
thus we produce a (typically many-to-one) map, $A\mapsto R(A)$, from $\A$ to $\R$.
We will then, for each $R\in \R$, fix some $A_R^*$ that maps to $R$, and take this
and the associated $G_R^*:= G_{A_R^*}$ to be our final approximations to $A$ and $G_A$
for each $A$ with $R(A)=R$.

The second phase of the argument then considers the intersections of $\XX$ with our
various pieces, as well as with the final bits that are added and subtracted to move
from the approximations to our actual $A$'s and $G_A$'s.
As suggested earlier, we hope to say that (w.h.p.) all these intersections have sizes
close to their expectations, and a central issue will be controlling the numbers of pieces
of various sizes:  the larger the pieces, the fewer we can afford.
This goal is achieved in Lemma~\ref{ML}, the workhorse of the second phase, from
which the desired application to \eqref{toshow4} follows fairly immediately:
see Corollary~\ref{MTpf} and the paragraph following its statement.

To get some feel for what's going on (in both phases) and how the whole thing fits together,
the reader might take an early look at the discussion of the
second phase through the proof of Corollary~\ref{MTpf},
ignoring the particulars of Lemma~\ref{ML} and reading the proof of the corollary
more or less at the level of Venn diagrams (without worrying about the meanings
of its many presently undefined ingredients).

\mn
{\em First phase.}
For $X\sub V:=V(\gS)$, let $N^i (X)=\{u\in V: \rho (u,X)\leq i\}$
(where, recall, $\rho$ is graph-theoretic distance).
For $A\in \A$ ($=\A(a,g)$), say a
path is $A$-{\em good} if it is of the form
$vx_1 y x_2$ with $x_1,x_2\in A$ (so in particular has length 3),
and for $v\in\gG_{k+1}$,
let $\vp(v,A)$ denote the number of $A$-good paths beginning
with $v$.

Fix a small $\gz>0$ (we just need $\gz<1/2$), and
set $\vt =\gz/2$ and
$$
G^0_A =\{v\in G_A : \vp(v,A)\geq  (1/4)k^{3-\gz}\}.
$$

\mn
For $T\sub\gG_k$ set $W_T=N^3 (T)\cap \gG_{k+1}$
and
\beq{ST}
S_T=\{ x: d_{W_T} (x)\geq k/2\}  ~~~(\sub \gG_k).
\enq
For $T\sub A\in \A$,
let $F_{A,T} =\nabla(N(T),\gG_k\sm A)$
and $Z_{A,T} = N(N^2(T)\cap A)\sub W_T$.
Notice that $w\in Z_{A,T}$ iff either $w\in N(T)$ or there is a path $xyzw$ with
$x\in T$ and $yz\not\in F_{A,T}$
(equivalently an $A$-good path from $w$ to $T$); in particular
$Z_{A,T}$ {\em is determined by $T$ and $F_{A,T}$.}

\begin{lemma}\label{Tlemma}
There is a fixed K such that
for each $A\in \A$
there is a $T\sub A$ satisfying

\mn
{\rm (T1)} $|T|\leq Kak^{-3+\gz}\log k$,

\mn
{\rm (T2)}
$|F_{A,T}|\leq K\gd a k^{-1+\gz}\log k$,

\mn
{\rm (T3)}
$|G^0_A \sm Z_{A,T}|\leq Kak^{-2}$,

\mn
{\rm (T4)}
$|W_T\sm G_A| < K\gd a k^\gz \log k$, and

\mn
{\rm (T5)}
$|A\sm S_T| < K\gd a k^{-\vt}$.
\end{lemma}
\nin
(The $\gz$ in the definition of $G_A^0$ is needed for the $\vt$ in (T5).
For the bound in (T5) we could actually get by with $O(\gd a\log^{-1} k)$;
see the discussion following \eqref{AsmSR} for more on this relatively delicate point
and \eqref{lastcount} for use of the bound.)

\medskip
The following auxiliary definitions and
lemma will be helpful in the proof of Lemma~\ref{Tlemma} and again later in the proof
of Lemma~\ref{ML}.
Fix $A\in \A$, set
$G_A=G$ and $G^0_A=G^0$, and define
$$H = \{y\in G : d_A(y)< k^{1-\vt}\},$$
$$B = \{x\in A : d_H(x)> k/2\},$$
$$I = \{y\in G\sm H : d_{A\sm B}(y)< k^{1-\vt}/2\}$$
and
$$C = \{x\in A\sm B : d_{H\cup I}(x)> k/4\}.$$

\begin{lemma}\label{G0lemma}
With the above definitions,
$|H\cup I|<O(\gd a)$, each of $|B|, |C| $ is $O(\gd ak^{-\vt})$,
and $G\sm G^0\sub H\cup I$.

\end{lemma}
\begin{proof}
We have
$$(k+1-k^{1-\vt})|H|\leq |\nabla (H,\gG_k\sm A)|\leq
|\nabla (G,\gG_k\sm A)|=(1+(k+1)\gd )a$$
(see \eqref{after10} for the equality),
$$(k/2) |B|<|\nabla(B,H)|< k^{1-\vt} |H|,$$
$$(k^{1-\vt}/2) |I|< |\nabla(I,B)|< k|B|/2$$
and
$$(k/4)|C|< |\nabla(C,H\cup I)| < |H\cup I|k^{1-\vt},$$
implying
$|H|<(4+o(1))\gd a$
(using \eqref{dxa''}),
$|B|<(8+o(1))\gd a k^{-\vt}$,
$|I|<(8+o(1))\gd a $
and
$|C|< (48+o(1))\gd a k^{-\vt}$.
This gives the first two assertions in the lemma.
The third is given by the observation that for
$y\in G\sm (H\cup I)$ the number of paths
$ywzx$ with $(w,z,x)\in (A\sm B)\times (G\sm H)\times A$
is at least $(k^{1-\vt}/2)(k/2)k^{1-\vt}$.
\end{proof}

\mn
{\em Proof of Lemma \ref{Tlemma}.}
Here we will find it more convenient to use
``big Oh" notation; that is, we will prove the lemma
with each of the bounds $K\cdot X$
appearing in (T1)-(T5) replaced by $O(X)$.
We first show existence of $T$ satisfying (T1)-(T3)
and then observe that any such $T$ also
satisfies (T4) and (T5).

Let $q=16 k^{-3+\gz}\log k$ and $\TT=A_q$
(the random subset of $A$ in which elements of $A$ appear independently,
each with probability $q$).
To show that there is a $T$ satisfying (T1)-(T3),
it is enough to
show that the stated bounds (again, in their ``big Oh" forms)
hold for the
{\em expectations} of
the set sizes in question,
since Markov's Inequality then implies existence of a $T$
for which each of these quantities is at
most three times its expectation.
This is of course true for  $\mathbb{E}|\TT|=aq$.
For (T2) we have (using \eqref{after10} for the final inequality)
\begin{eqnarray*}
\mathbb{E} |F_{A,\TT}| &=& \sum_{x\in G}\mathbb{P}(x\in N(\TT))d_{\gG_k\sm A}(x)\\
&\leq & q\sum_{x\in G}d_A(x)d_{\gG_k\sm A}(x) \\
&\leq & qk|\nabla (G,\gG_k\sm A)| <O(\gd a k^{-1+\gz}\log k).
\end{eqnarray*}

To bound the expectation for
(T3), notice
that for $v\in G^0$, there are at least $(1/8)k^{3-\gz}$
vertices $x\in A$ for which $x\in T$ implies $v\in Z_{A,T}$.
(This is true of any $x$ for which there is an $A$-good path
from $v$ to $x$ and,
since two vertices at distance 3 are connected by exactly
two paths of length 3 in $\gS$,
the number of such $x$'s is at least $\vp(v,A)/2$.)
The probability that such a $v$ does not belong to
$Z_{A,\TT}$ is thus
at most $(1-q)^{(1/8)k^{3-\gz}}< k^{-2}$, so that
$\mathbb{E}|G^0\sm Z_{A,\TT}|< gk^{-2}$
(which gives the bound in (T3) since we assume $g=O(a)$;
of course the assumption isn't really needed here, as
we could instead have arranged (e.g.) $\mathbb{E}|G^0\sm Z_{A,\TT}|< gk^{-3}$).

This completes the discussion of (T1)-(T3) and we turn to
the last two properties requested of $T$.
We first observe that (T4) follows from (T2), since in fact
\beq{WTsm}
|W_T\sm G|\leq k|F_{A,T}|.
\enq
To see this just notice that
if $w\in W_T\sm G$, then (since $w\in W_T$)
there is a
path $xyzw$ with $x\in T$ and (therefore)
$y\in N(T)$, but
$z\not\in A$ (since $w\not\in G$), so that
$yz\in F_{A,T}$ (and each such $yz$ gives rise to at most $k$ such $w$'s).

For (T5), note that (according to
the definition of $S_T$ in \eqref{ST})
any $x\in A\sm S_T$ has at least $k/4$ neighbors in
one of $G\sm G^0$, $G^0\sm W_T$.
By Lemma \ref{G0lemma}, $x$'s of the first type belong to $B\cup C$
and number at most
$O(\gd ak^{-\vt})$.
On the other hand, by (T3) (and \eqref{dxa''}),
the number of the second type is at most
$$
(4/k)|G^0\sm W_T|(k+1) < O(ak^{-2})
<o(\gd a k^{-\vt}).
$$
\qed

We think of $W_T$ in Lemma~\ref{Tlemma} as a first
approximation to $G_A$, and $Z_{A,T}$ as a second approximation
satisfying
\beq{WZA}
Z_{A,T}\sub W_T\cap G_A
\enq
that discards vertices that got into $W_T$ on spurious
grounds.
Similarly, the next lemma prunes our first approximation,
$S_T$, of $A$ to get a better second approximation.

\begin{lemma}\label{Ulemma}
There is a fixed K such that
for any $A\in \A$ and $T\sub A$ satisfying
(T4), there is some $U\sub W_T\sm G_A$
with

\mn
{\rm (U1)} $|U|\leq K\gd a k^{-1+\gz} \log^2 k$ and

\mn
{\rm (U2)}
$|(S_T\sm A)\sm N(U)|\leq K\gd a$.
\end{lemma}

\nin
The second approximation mentioned above is then
$S_T\sm N(U)$, which in particular satisfies
\beq{S'}
S_T\supseteq S_T\sm N(U)\supseteq S_T\cap A.
\enq

\mn
{\em Proof of Lemma \ref{Ulemma}.}
Here we again (as in the proof of Lemma \ref{Tlemma})
switch to ``big Oh" notation.
Set $G=G_A$, $W=W_T$ and $S=S_T$.
Let $q=4k^{-1}\log k$ and $\UU=(W\sm G)_{q}$.
By the definition of $S=S_T$, each
$x\in S\sm A$ has at least $k/4$ neighbors in one of
$W\sm G$, $G$.
Let
$$L=\{ x\in S\sm A : d_{W\sm G} (x)\geq k/4\}.$$
Then
$|L|\leq (4/k)|W\sm G|(k+1)
=O(\gd a k^{\gz}\log k)$ (by (T4)).
On the other hand, for $x\in L$ we have
$\mathbb{P} (x\not\in N(\UU))\leq (1-q)^{k/4}< k^{-1}$,
so there is some $U$
with
\[
|L\sm N(U)|\leq
\mathbb{E} |L\sm N(\UU)| <|L|/k=o(\gd a).
\]
Finally, since $x\in (S\sm A)\sm L$ implies
$d_G(x)>k/4$, we have
$$|(S\sm A)\sm L|
\leq (4/k)|\nabla (G,\gG_k\sm A)|=4(1+(k+1)\gd )a/k
=O(\gd a).$$
The lemma follows.
\qed

\medskip
Now write $K$ for the larger of the constants appearing in
Lemmas~\ref{Tlemma} and \ref{Ulemma}.
For each $A\in \A$ fix some $T=T_A\sub A$ satisfying
(T1)-(T5)
and then some $U=U_A\sub W_T\sm G_A$ satisfying
(U1)-(U2), and set:
$W_A=W_T$, $S_A=S_T$,
$F_A=F_{A,T}$, $Z_A=Z_{A,T}$,
$S'_A= S_T\sm N(U)$ and
$\rrr_A = \rrr(A)= (T_A,F_A,U_A)$.
(We prefer $R_A$ but will sometimes use $\rrr(A)$ to avoid double subscripts.)
We may
think of $T_A,F_A,U_A$ as ``primary" objects, which
we need to specify, and
$W_A,S_A,Z_A,S'_A$ as ``secondary" objects,
which are functions of the primary objects.

Let $\R=\{\rrr_A:A\in \A\}$.
If $\rrr=\rrr_A$ then we also
set $W_\rrr=W_A$ (which is the same for all
$A$ with $\rrr_A=\rrr$), and similarly for
the other objects subscripted by $A$ in the preceding paragraph.
For each $\rrr\in \R$ fix some $A^*=A^*_\rrr\in \A$ with
$\rrr_{A^*}=\rrr$, and let $G^*_\rrr=G_{A^*}$.
Now suppose $A\in \A$, $G=G_A$, $R=R_A$, $A^*=A^*_R$ and
$G^*=G_R^*$.  Notice that, given $A^*$ and $G^*$,
\beq{Adet}
\mbox{{\em $A$ is determined by $A\sm A^*$ and $G\cap G^*$.}}
\enq
(Actually a closed $A\sub \gG_k$ with $G=G_A$
is determined by $B$, $G_B$,
$A\sm B$ and $G\cap G_B$ for {\em any}
$B\sub \gG_k$, since
$$A\cap B = \{x\in B:N(x)\sub G\cap G_B\};$$
namely, $x\in A$ iff $N(x)\sub G$,
which for $x\in B$ is the same as $N(x)\sub G\cap G_B$.)

\mn
{\em Second phase.}
We now turn to $\XX$.  In what follows we
assume the constant $\eps$ ($=1-p$) is small enough to support our argument,
making no attempt to optimize.

For $\eta>0$ and $B\sub V(\gS)$
(we will always have $B\sub \gG_k$ or $B\sub \gG_{k+1}$), we will be interested in the event
\beq{EBeta}
E_{B,\eta}=\{||\XX\cap B|-|B|p|> \eta \gd a p\}.
\enq
(The second $p$ on the right-hand side is unnecessary
but we keep it as a
reminder of where we are:  if $p$ were smaller,
then this factor {\em would} be relevant.)
Say a collection $\B$
of sets is $\eta$-{\em nice} if
\beq{etanice}
\mathbb{P}(\cup_{B\in \B}E_{B,\eta}) < \exp[-\gO(ak^{-2})].
\enq

Fix a smallish $\eta$; for concreteness,
say $\eta = 0.08$
(we need $6\eta < 0.5$).
The next, regrettably (but as far as we can see unavoidably)
elaborate statement
is most of the story.

\begin{lemma}\label{ML}
The following collections are $\eta$-nice:

\mn
{\rm (a)}
$\{W_\rrr:\rrr\in \R\}$;

\mn
{\rm (b)}
$\{S_\rrr:\rrr\in \R\}$;

\mn
{\rm (c)}
$\{W_\rrr\sm Z_\rrr:\rrr\in \R\}$;

\mn
{\rm (d)}
$\{S_\rrr\sm S'_\rrr:\rrr\in \R\}$;

\mn
{\rm (e)}
$\{S'_\rrr\sm A^*_\rrr:\rrr\in \R\}$;

\mn
{\rm (f)}
$\{A^*_\rrr\sm S'_\rrr:\rrr\in \R\}$;

\mn
{\rm (g)}
$\{G^*_\rrr\sm Z_\rrr:\rrr\in \R\}$;

\mn
{\rm (h)}
$\{A\sm A^*_{\rrr(A)}:A\in \A\}$;

\mn
{\rm (i)}
$\{A^*_{\rrr(A)}\sm A:A\in \A\}$;

\mn
{\rm (j)}
$\{G_A\sm G^*_{\rrr(A)}:A\in \A\}$;

\mn
{\rm (k)}
$\{G_A\cap  (G^*_{\rrr(A)}\sm Z_{\rrr(A)}):A\in \A\}$.
\end{lemma}

\mn
Before proving this, we show that it supports
\eqref{toshow4}:
\begin{cor}\label{MTpf}
The collections $\A$ and $\{G_A:A\in \A\}$ are $(6\eta)$-nice.
\end{cor}

\nin
Of course this gives the relevant portion of \eqref{toshow4}, since $\Q(a,g)$ implies
that for some $A\in \A$ either
$|\XX\cap A|\geq |A|p+\gd a p/2$ or
$|\XX\cap G_A|\leq |G_A|p-\gd a p/2$
({\em cf.} \eqref{EXGA}), each of which, according to Corollary~\ref{MTpf},
occurs with probability $\exp[-\gO(ak^{-2})]$ (and---recall \eqref{alarge}---
\[
\sum_{a>a_0}\sum_{g\leq 2a}\exp[-\gO(ak^{-2})] =o(1/n)).
\]

\bn
{\em Proof of Corollary \ref{MTpf}.}
This is just a matter of building the relevant sets,
starting from the collections in Lemma \ref{ML}
and applying the (trivial) observations:

\mn
if $\{K_B:B\in \B\}$ is $\ga$-nice, $\{L_B:B\in \B\}$
is $\gb$-nice
and $K_B\cap L_B=\0 ~~\forall B\in \B$, then
$\{K_B\cup L_B:B\in \B\}$ is $(\ga+\gb)$-nice;

\mn
if $\{K_B:B\in \B\}$ is $\ga$-nice, $\{L_B:B\in \B\}$
is $\gb$-nice
and $K_B\supseteq L_B ~~\forall B\in \B$, then
$\{K_B\sm L_B:B\in \B\}$ is $(\ga+\gb)$-nice.

\mn
Using these (in combination with Lemma
\ref{ML}), we find that:

\[
\mbox{$\{Z_\rrr=
W_\rrr\sm (W_\rrr\sm Z_\rrr):\rrr\in \R\}$ is $(2\eta)$-nice;}
\]
\[
\mbox{$\{S'_\rrr=S_\rrr\sm (S_\rrr\sm S'_\rrr):\rrr\in \R\}$
is $(2\eta)$-nice;}
\]
\[
\mbox{$\{A^*_\rrr=(S'_\rrr\sm (S'_\rrr\sm A^*_\rrr))\cup
(A^*_\rrr\sm S'_\rrr):\rrr\in \R\}$ is $(4\eta)$-nice;}
\]
\[
\mbox{$\{A=
(A\sm A^*_{\rrr(A)})\cup (A^*_{\rrr(A)}\sm (A^*_{\rrr(A)}\sm A)):
A\in \A\} =\A$ is $(6\eta)$-nice;}
\]
\[
\mbox{$\{G_A= (G_A\sm G^*_{\rrr(A)})
\cup (G_A\cap (G^*_{\rrr(A)}\sm Z_{\rrr(A)}))\cup Z_{\rrr(A)}):
A\in \A\}$}~~~~~~~
\]
\[
~~~~~~~~~~~~~~~~~~~~~~~~~~~~~~~~~~~~~~~~~~~~~~~
\mbox{$=\{G_A:A\in \A\}$ is $(4\eta)$-nice.}
\]
(Note $Z_{\rrr(A)}$ is the same as $Z_{A}$ but seems slightly more natural here.)
\qed


\bn
{\em Proof of Lemma \ref{ML}.}
For the rest of this discussion we
write $E_B$ for $E_{B,\eta}$.
We want to show that
\eqref{etanice}
holds for each of the collections---say $\B$---appearing in (a)-(k).
This is all based on the union bound:
in each case we bound the size of the $\B$ in question
and show, using
what we know about the sizes of
members of $\B$, that $\mathbb{P}(E_B)$
is much smaller than $|\B|^{-1}$ for each $B\in \B$.

We are interested in bounding probabilities of the type
\[
\mathbb{P}(||\XX\cap B| -|B|p| > \eta \gd ap)
\]
using Theorems~\ref{Chern} and \ref{uppertail};
but, since $p=1-\eps\approx 1$, we can
do a little better by applying these theorems with
$\xi = |B\sm \XX|$
(which has the distribution ${\rm Bin}(|B|,\eps)$),
using the trivial observation
that, for any $\gl>0$ (always equal to
$\eta \gd ap$ in what follows),
\beq{XminusB}
\mathbb{P}(||\XX\cap B| -|B|p| > \gl)
=\mathbb{P}(||B\sm \XX| -|B|\eps| > \gl).
\enq
(For most of the argument this change will make little difference,
but it will be crucial
when we come to items (h)-(k).)

\mn
{\em
Items
(a) and (b)}.
To make things easier to read, set $\bbb=Kak^{-3+\gz}\log k$
(the bound in (T1) of Lemma \ref{Tlemma}).
The number of possibilities for each of $W_\rrr$, $S_\rrr$
is bounded by the number of possible $T\rr$'s, which
(by \eqref{Cnt}) is at most
\beq{expb}
\exp[\bbb\log (eN/\bbb)] < \exp[\bbb \log (Nk^3/a)]
\enq
(recall $N=\C{2k}{k}$).
On the other hand, (T4)
and the fact that
$|S_T|\leq 2(k+1)|W_T|/k$ (see the definition of $S_T$ in \eqref{ST})
imply that, for any $T$,
\[   
|W_T|, |S_T|< O(\gd a k^{\gz}\log k +g) =
O(\gd a k^{\gz}\log k +a),
\]   
so that Theorem \ref{Chern} gives (for any $T$)
\beq{XWT}
\max\{\mathbb{P}(E_{W_T}), \mathbb{P}(E_{S_T}) \}<
\exp[-\gO (\gd^2a/(\gd k^\gz\log k+1))].
\enq
(In a little more detail:
we apply Theorem \ref{Chern}
(using \eqref{XminusB}, though, as noted above,
this is not really needed here), with
$m = O(\gd a k^{\gz}\log k +a)$, $q=\eps$ and
$\gl = \eta \gd ap$,
to bound the left side of \eqref{XWT} by
$\exp[-\gO(\gl^2/\max\{m\eps,\gl\})]$,
and observe that
$\max\{m\eps,\gl\}
=O(\gd ak^\gz\log k+a)$.)

That the collections in (a) and (b) are $\eta$-nice now follows
upon multiplying the bounds in
\eqref{expb} and \eqref{XWT} and checking that \eqref{dxa''}
implies (with room to spare)
$
\bbb \log (Nk^3/a) =o(\gd^2a/(\gd k^\gz\log k+1)).
$

\mn
{\em Item (c)}.
Since each of $Z_\rrr$, $W\rr$ is determined by $T\rr$
and $F\rr$,
the number of possibilities for $W\rr\sm Z\rr$
is at most the product
of the bound in \eqref{expb}
(which will be negligible here) and the number of possibilities
for $F_\rrr$ given $T:=T_\rrr$.
The latter is at most the number of
subsets of $\nabla(N(T),\gG_k\sm T)$ of size less than
$\ccc:=K\gd ak^{-1+\gz}\log k$ (the bound in (T2)),
which, since
\beq{nablaTgG}
|\nabla(N(T),\gG_k\sm T)| \leq k^2|T|< Kak^{-1+\gz}\log k=:\ddd
\enq
(see (T1)), is less than
\begin{eqnarray}
\exp [\ccc \log (e\ddd/\ccc)]
&=&\exp[O(\gd ak^{-1+\gz}\log k\log(e/\gd))]\nonumber\\
&=&\exp[O(\gd ak^{-1+\gz}\log^2 k)].\label{ogdak}
\end{eqnarray}
(Here we again use \eqref{Cnt} (for the initial bound) and
\eqref{dxa''} (for the second line).
Strictly speaking, the application of \eqref{Cnt} is only justified
when $\gd\leq 1/2$;
but for larger $\gd$ we can bound the number of possibilities for $F_R$ by the trivial
$2^{\ddd}$,
which (for such $\gd$) is smaller than the left side of \eqref{ogdak}.)

On the other hand, again using (T2), we have
\[
|W_R\sm Z_R|\leq k|F_R| =
O(\gd a k^{\gz}\log k)
\]
(with the justification for the first inequality similar to that for \eqref{WTsm}).

\medskip
Thus Theorem~\ref{Chern} gives (for any $\rrr$)
\begin{eqnarray}
\mathbb{P}(E_{W\rr\sm Z\rr}) &<&
\exp[-\gO (\eta^2\gd^2a^2/(\gd a k^\gz\log k))]\nonumber\\
&=&\exp[-\gO (\eta^2\gd a/ (k^\gz\log k))],\label{22A}
\end{eqnarray}
which, combined with the (here insignificant)
bounds in \eqref{expb} and \eqref{ogdak}, gives
$$
\mbox{$\sum_\rrr\mathbb{P}(E_{W\rr\sm Z\rr})
=\exp[-\gO (\eta^2\gd a/ (k^\gz\log k))].$}
$$

\mn
{\em Items (d)-(g).}
For each of these the number of sets in question
is $|\R|$, the number of possibilities
for $(T\rr,F\rr,U\rr)$.
As already observed, the number of $(T\rr,F\rr)$'s is at
most the product of the bounds in \eqref{expb} and \eqref{ogdak}.
On the other hand, with
$\ccc= K\gd a k^{-1+\gz} \log^2 k$
(the bound on $|U|$ in (U1)) and
$\ddd =K\gd a k^{\gz}\log k$
(the bound on $|W_T\sm G_A|$ in (T4))---so
$\ccc$ and $\ddd$ have changed from what they were above---the
number of possibilities for $U_R$ given $T_R$ is at most
\beq{Ubd}
\exp [\ccc \log (e\ddd/\ccc)]
=\exp[O(\gd a k^{-1+\gz} \log^3 k)]
\enq
(which dominates the bounds from \eqref{expb} and \eqref{ogdak}).

We next need to bound the sizes of the various sets
under discussion.
We have

\beq{SAS'A}
|S_R\sm S'_R|\leq (k+1)\ccc =O(\gd ak^\gz\log^2k)
\enq
(again, since (U1) bounds $|U_R|$ by $\ccc$ (and $S_R\sm S_R'\sub N(U_R)$));
\beq{Srr'}
|S_\rrr'\sm A^*_\rrr| =O(\gd a)
\enq
(given by (U2), once we recall that $S_R'=S_R\sm N(U_R)$);
\beq{A*rr}
|A^*_\rrr\sm S'_\rrr| =O(\gd a k^{-\vt})
\enq
(using (T5) and the fact---see \eqref{S'}---that
$A_R^*\sm S_\rrr' =A_R^*\sm S_\rrr$); and, with $A_R^*=A$
(so $G_R^*=G_A$),
\beq{G*rr}
|G^*_\rrr\sm Z_\rrr|\leq |G_A^0\sm Z_\rrr|
+|G^*_\rrr\sm G_A^0|
= O(ak^{-2}+ \gd a) =O(\gd a)
\enq
(using (T3), Lemma \ref{G0lemma} and \eqref{dxa''}.)
Note, for use below, that
for {\em any} $A$ with $R(A)=R$,
\eqref{G*rr} remains true if we replace $G_R^*$ by $G_A$.
(Similarly \eqref{A*rr} holds with any such $A$ in place of $A_R^*$, but we
don't need this.)

The largest of the bounds in \eqref{SAS'A}-\eqref{G*rr}
is
the $O(\gd a k^\gz\log^2 k)$ in \eqref{SAS'A}; so for
each of the sets $B$ appearing in (d)-(g)
(i.e. $B = S_\rrr\sm S_\rrr'$ in (d) and so on),
we have
\begin{eqnarray}
\mathbb{P}(E_B) &<&
\exp[-\gO (\eta^2\gd^2a^2/(\gd ak^\gz\log^2 k)]\nonumber\\
&=&\exp[-\gO (\eta^2\gd a/ (k^\gz\log^2 k))]\label{EQ};
\end{eqnarray}
and, since $\eta^2\gd a/ (k^\gz\log^2 k)$
in \eqref{EQ} is much larger than
the exponent in \eqref{Ubd}, it follows that the
collections in (d)-(g) are $\eta$-nice.

\mn
{\em Items (h)-(k).}
Here we first dispose of the sizes of the individual sets,
before turning to
the more interesting problem of
bounding the sizes of the collections
in question.

For (h) and (i), notice that for any $A,A'\in \A$ with
$R(A)=R(A')$ we have
$$
|A\sm A'|\leq |A\cap (S'_\rrr\sm A')| +
|A\sm S'_\rrr| = O(\gd a + \gd a k^{-\vt})
=O(\gd a)
$$
(using (U2) and (T5), as earlier in
\eqref{Srr'} and \eqref{A*rr}); in particular this bounds
the sizes of the sets in (h), (i)
(namely $|A\sm A^*_\rrr|$ and
$|A^*_\rrr\sm A|$, where $R=R(A)$) by $O(\gd a)$.
For (j) and (k), a similar bound---that is,
\[
\max\{|G_A\sm G_\rrr^*|, |G_A\cap  (G^*_{\rrr}\sm Z_{\rrr})|\} =O(\gd a)
\]
for $A$ with $R(A)=R$---follows from
$G_R^*\supseteq Z_R$ (see
\eqref{WZA}), $|G_A\sm Z_R|< O(\gd a)$ (noted following \eqref{G*rr})
and \eqref{G*rr} itself.

\medskip
We now turn to the sizes of the collections in (h)-(k),
each of which is at most $|\A|$.  We will show
\beq{|A|}
|\A|<\exp[O(\gd a)].
\enq
Before doing so we observe that this is enough
to show that the collections in (h)-(k) are $\eta$-nice;
namely, for $B$ belonging to any of these collections
(so $|B|=O(\gd a)$) and small enough $\eps$,
Theorem \ref{uppertail} (applied with
$m=O(\gd a)$ and $q=\eps $---and now really using \eqref{XminusB}---gives
\beq{22B}
\mathbb{P}(E_B) < \exp[-\eta \gd a(\log (1/\eps)-O(1))].
\enq
(Here
$|B\sm \XX|< |B|\eps -\eta\gd ap$ is impossible,
so we are just using
\[
\mathbb{P}(|B\sm \XX|> |B|\eps +\eta \gd ap) <
\mathbb{P}(|B\sm \XX|> \eta \gd ap) <
\exp[-\eta \gd a p\log (\tfrac{\eta \gd ap }{em\eps})].)
\]
Assertions (h)-(k) then follow on multiplying the bounds in \eqref{|A|}
and \eqref{22B}.

\bn
{\em Proof of \eqref{|A|}.}
According to
\eqref{Adet}, we may bound $|\A|$
by the number of possibilities for the pair
$(A\sm A_\rrr^*,G_A\cap G_\rrr^*)$
(with $R_A=R$), so by
our earlier bound
on $|\R|$---essentially that in \eqref{Ubd}; see the discussion
of items (d)-(g)---multiplied by the number of
possibilities for $(A\sm A_\rrr^*,G_A\cap G_\rrr^*)$ given $\rrr$.
So it is enough to show that, once we know $R$---and therefore
$A_R^*$ and $G_R^*$---the number of choices for each of
$A\sm A_\rrr^*$, $G_A\cap G_\rrr^*$
is less than $\exp[O(\gd a)]$.

The second of these is easy:
since (by \eqref{WZA})
each of $G_A,G_R^*$ contains $Z_R$
(which is determined by $R$), the number of possibilities
for $G\cap G_\rrr^*$ given $R$ (and therefore $G_R^*$)
is at most
$\exp_2[|G_R^*\sm Z_R|]$,
and we have already seen in \eqref{G*rr} that
$|G_\rrr^*\sm Z_\rrr|=O(\gd a)$.

The case of
$A\sm A_\rrr^*$ is more interesting.
Here we decompose
\[
A\sm A_\rrr^* = (A\cap (S'_\rrr\sm A^*_\rrr))
\cup (A\sm (S_\rrr'\cup A^*_\rrr))
\]
and consider the two terms on the right-hand side
separately.
The number of possibilities for the first term is
at most $\exp_2[|S_R'\sm A_R^*|]$
(again, given $R$, which determines $S_R'$ and $A_R^*$),
while (U2) (or \eqref{Srr'}) gives $|S_R'\sm A_R^*|=O(\gd a)$.

So it is enough to show that the number of possibilities
for $A\sm (S_\rrr'\cup A^*_\rrr)$ is $\exp[O(\gd a)]$
(it will actually be much smaller).
In fact, it is enough
to prove such a bound on
the number of possibilities for
$A\sm S_\rrr'$, which
determines $A\sm (S_\rrr'\cup A^*_\rrr)$
since we know $A^*_\rrr$.
Here we recall that
\eqref{S'} gives
$A\sm S_\rrr'=A\sm S_\rrr$ (so we may use these interchangeably,
and similarly for $A\cap S_\rrr'=A\cap S_\rrr$), and that---crucially---(T5) gives
\beq{AsmSR}
|A\sm S_R|=O(\gd a k^{-\vt}).
\enq

Note that this final point differs from its
earlier counterparts in that we now have less
control over the size of the universe from
which the set in question (i.e. $A\sm S_R$) is being drawn
(in contrast to, for example, $F_R$ in (c), which was
drawn from
$\nabla(N(T),\gG_k\sm T)$,
whose size was bounded in \eqref{nablaTgG},
or, in the present case, $A\cap (S_R'\sm A_R^*)$, which is
drawn from the quite small $S_R'\sm A_R^*$).
Thus, for example, if we try to apply \eqref{Cnt} with
$\aaa$ the bound in \eqref{AsmSR} and $\bbb =N$
($=\C{2k}{k}$), then
we can only say that
the number of possibilities
for $A\sm S_R$ is less than
$\exp[O(\gd a k^{-\vt})\log (eN/\gd a k^{-\vt})]$,
which for somewhat small $a$ may be far larger than
the desired $\exp[O(\gd a)]$.
This little difficulty will be handled by
Proposition \ref{Knuth}.

\medskip
Write $\ttt$ ($=O(\gd a k^{-\vt})$)
for the bound on
$| A\sm S_\rrr|$ given in \eqref{AsmSR}.
Denote by $\gL$ the (``Johnson") graph on $\gG_k$ in which
two vertices (a.k.a. $k$-sets)
are adjacent if they are
at distance 2 in $\gS$, and set $d=k^2$ (so
$\gL$ is $d$-regular).
Since our $A$'s induce
connected subgraphs of $\gL$
(another way of saying they are 2-linked),
there is, for each $A$ under discussion, a rooted
forest
with roots in $S_R\cap A = S'_R\cap A$,
set of (at most $t$) non-roots equal to $A\sm S_R$,
and at least one non-root in each component;
thus we just need to bound the
number of such forests.

(Note that existence of said forest requires $S_R\cap A\neq\0$, which, since we assume $\gd$ is not too large,
holds because
the bound in (T5) is less than $a$.
If $S_R\cap A=\0$, as is possible for large $\gd$, we may bound the number
of choices for $A\sm S_R=A$ by the number of {\em trees} of size up to $t$,
but this count should include a factor $N$ (in place of the bound for (ii) below)
for the choice of a root---a change that can cause trouble in the present
regime, but not for large $\gd$, where, as will appear below, our probability bounds improve.)

For the desired bound we may think of specifying a forest as above
by specifying:

\mn
(i)  the number, say $q\leq t$, of roots;

\mn
(ii)  the set of roots, $\{x_1\dots x_q\}\sub S_R'\cap A$;

\mn
(iii)  for each $i\in [q]$, the size, say $\ga_i$,
of the component (tree) rooted at $x_i$; and

\mn
(iv)  the components themselves.

\mn

\mn
We may bound the numbers of possibilities in (ii), (iii) and
(iv) by
$\C{a+O(\gd a)}{q}$,
$\C{t}{q}$ and $(ed)^t$ respectively.  The first of these derives
from (U2), according to which we have $|S_R'|<a +O(\gd a)$;
the second is the number of sequences $(\ga_1\dots \ga_q)$
of positive integers summing to at most $t$;
and the third is given by Proposition \ref{Knuth}.
Thus (recalling from \eqref{alarge} that $a$ is not very small),
we find that the number of forests as above is at most
\beq{lastcount}
\mbox{$\sum_{q\leq t}
\Cc{a+O(\gd a)}{q}\Cc{t}{q}(ed)^t =
\exp[\Theta(t\log k)] ~~~(=\exp[O(\gd a)]$).}
\enq

\qed

Finally we turn to the case of large $\gd$ ($\gd > 1$),
showing (for any $a,g$, with $\gd = (g-a)/a> 1$)
\beq{toshowlast}
\mathbb{P}(\Q(a,g)) < \eps^{g/3},
\enq
which, with the trivial $g\geq k$, bounds the contribution to
\eqref{toshow4} of the terms under discussion by
\[
\sum_{g\geq k}\sum_{a< g}\eps^{g/3}=o(1/n).
\]

\medskip
For \eqref{toshowlast}, first notice that in the present situation Theorem~\ref{uppertail}
bounds the probability of
\eqref{XGA} (for a given $A\in \A(a,g)$)
by
\beq{2eeps}
\mathbb{P}(|G_A\sm \XX|> g/2) < (2e\eps)^{g/2}.
\enq

On the other hand, to bound the number of possibilities for $A$
(i.e. the size of $\A(a,g)$), we may think of specifying $A$ {\em via}
the following steps.

\mn
(i)  Choose, for an appropriate fixed $C$, $T\sub G:=G_A$ of size $(C\pm o(1))(g/k)\log k$ such that,
with
\[
S=S_T=\{x\in \gG_k:d_T(x)> (C/2)\log k\},
\]
we have
\beq{AminusS}
|A\sm S|< k^{-2}a
\enq
and, with
$Z=Z(G) = \{x\in \gG_k: d_G(x)\geq k/4\}$,
\beq{Ssize}
|S\sm Z| < k^{-1}g.
\enq
(The proof of the existence of such a $T$ is similar to---easier than---the proof of
Lemma~\ref{Tlemma}, and we omit the details, just noting that, since $S\sub N(G)$,
fewer than $gk$ vertices are candidates for $S\sm Z$.)

Notice that by \eqref{Ssize} (and the definition of $Z$), we have
\beq{Ssize'}
|S|\leq (4/k)g(k+1) + g/k    =O(g).
\enq

\mn
(ii)  For each $x\in A\sm N(T)$ ($\sub A\sm S$), choose some neighbor of $x$
(necessarily in $G$) and
let $T'$ be the collection of these neighbors;
thus $T'\cap T=\0$ and
$|T'|\leq |A\sm S|< k^{-2}a$ (by \eqref{AminusS}).
Notice also that $T\cup T'$ is 4-linked (by Proposition~\ref{Plink}
and the fact that $A$ is 2-linked).

\mn
(iii)  Finally, choose $A$ from $S\cup N(T')$.

\medskip
We should then bound the number of ways in which these steps can be carried out:

\mn
(i)  Since $T\cup T'$ is 4-linked, Proposition~\ref{Knuth}
(applied to the graph on $\gG_k$ in which adjacency is $\gS$-distance at most 4,
so a $d$-regular graph for some $d<k^4$) bounds the number of choices for
$T\cup T'$ by
\[
M\exp[O( (g/k)\log k)\log d]<M\exp[O (g/k)\log^2 k]
\]
(where the $M$ ($=|\gG_{k+1}|$) corresponds to choosing a root in $T\cup T'$).

\mn
(ii)  The number of choices for $T'$ given $T\cup T'$ is
$\exp[O(k^{-2}a \log (gk/a)]$.
Note that once we know $T\cup T'$ and $ T'$, we also know $T$ and
thus $S$.

\mn
(iii)  Given $T'$ and $S$, there are at most $\exp[a( \log (g/a)+O(1))]$ choices for $A\sub S\cup N(T')$
(since $|S\cup N(T')| <O( g)$; see \eqref{Ssize'} and the specification of $|T|$ in (i)).

\medskip
Of course for sufficiently (not very) small $\eps$,
all of these bounds are dominated by the one in
\eqref{2eeps}, so we have \eqref{toshowlast}.
\qed

\section{Sperner}

As one might expect, there has also been some consideration of
Sperner's Theorem \cite{Sperner}---usually considered the first result in
extremal set theory---from the sparse random viewpoint.
Here $\XX$ is the random subset of $2^{[n]}$ in which each $A\sub [n]$ is present
with probability $p$ (independent of other choices), and
one is interested in the size of a largest antichain
(collection of pairwise incomparable sets)
in $\XX$ (called the {\em width} of $\XX$ and denoted $w(\XX)$; see
\cite{Bollobas} for general background and
\cite{BMT} for a review of work related to the present question).
In particular, proving a conjecture of Osthus \cite{Osthus},
it is shown in \cite{BMT} and \cite{CM}
(both using the ``container" technology of
\cite{BMS,ST}) that
$
w(\XX)\sim \Cc{n}{\lfloor n/2\rfloor}p
$
w.h.p. provided $p>C/n$ for a suitable fixed $C$.

\medskip
Here again it is natural to ask for a more literal counterpart of Sperner's Theorem,
namely, for the property
\beq{wwXX}
w(\XX) = \max\{|\XX\cap \gG_{\lfloor n/2\rfloor}|,|\XX\cap \gG_{\lceil n/2\rceil}|\}
\enq
(where we again take $\gG_\ell =\C{[n]}{\ell}$).
As for the Erd\H{o}s-Ko-Rado question considered above (and for similar reasons),
it is easy to see that \eqref{wwXX} is unlikely for $p$ less than about 3/4
(and easy to guess that it {\em is} likely above this).
Here we just observe that the method of Section~\ref{Proof}
at least gives the weaker statement analogous to our Theorem~\ref{MT}:
\begin{thm}\label{Sthm}
There is a fixed $\eps>0$ such that \eqref{wwXX} holds w.h.p. provided $p>1-\eps$.
\end{thm}
\nin
(Though not in print as far as we know, this seems to have been
of some interest; the present authors first heard the question in a lecture of
J. Balogh \cite{Balogh}.)

We just indicate how this goes.
The main point is that the argument of Section~\ref{Proof} is easily adapted to show that
(for $\eps,p$ as in Theorem~\ref{Sthm}) w.h.p.
\beq{X1}
|\XX\cap \partial_u(A)|> |\XX\cap A|
\enq
whenever $A\sub \gG_i$ is closed and nonempty and
either $i< \lfloor n/2\rfloor$ or $n=2k+1$, $i=k$
and $|A|\leq \frac{1}{2}\C{n}{k}$,
and (with $\partial_l$ denoting lower shadow)
\beq{X2}
|\XX\cap \partial_l(B)|> |\XX\cap B|
\enq
whenever $B\sub \gG_i$
is closed and nonempty and
either $i> \lceil n/2\rceil$ or $n=2k+1$, $i=k+1$
and $|B|\leq \frac{1}{2}\C{n}{k}$.
(Note \eqref{X1} and \eqref{X2} hold in the specified regimes
provided they do so when $A$ and $B$ are
2-linked.)

\mn
[We have preferred not to extend the
material of Section~\ref{Proof} to cover the present situation,
feeling that the extra generality would make the argument
even harder to follow than it already is.
It should at least be intuitively clear that \eqref{X1} and \eqref{X2}
are in fact less delicate than what's gone before; e.g.
\eqref{X1} gets easier as $i$ shrinks
(with $n$ fixed;
so the hardest case would be $i=n/2$, which corresponds to what we did
earlier and does not even appear here).

For the warier reader we may also argue as follows (for \eqref{X1} say).
Given $i<n/2$, identify $2^{[n]}$ in the natural way with
$\{B\sub [n]\cup J:B\supseteq J\}$, where $J$ is some $(n-2i)$-set
disjoint from
$[n]$.  Our $\gG_i$ then becomes a subset of $\C{[n]\cup J}{k}$,
where $k=n-i = |[n]\cup J|/2$, and the results of Section~\ref{Proof} apply directly.
We will not elaborate, apart from noting that (i) in this case
the lower bound on
$\gd$ in \eqref{dxa''} is automatic, and
(ii) the need to sum failure probabilities
over possible values of $i$ causes no trouble since the
bounds on these probabilities (essentially those in
\eqref{XWT},\eqref{22A},\eqref{EQ},\eqref{22B},\eqref{2eeps})
are so small.]

\medskip
It remains to observe that \eqref{X1} and \eqref{X2} (for the stated ranges)
imply \eqref{wwXX}.
(They actually imply that
$\XX\cap \gG_{\lfloor n/2\rfloor}$ and $\XX\cap \gG_{\lceil n/2\rceil}$ are
the only possible
antichains of size $w(\XX)$;
so if $n$ is odd, then w.h.p.
one of these is the {\em unique} largest antichain,
since their sizes differ w.h.p.)

For $n$ even the implication is immediate:
if $\cup A_i$ is an antichain of $\XX$ with $A_i\sub \gG_i$
and $i=\min\{j:A_j\neq \0\}< n/2$, then replacing $A_i$
by $\XX\cap \partial_u (A_i)$ gives a larger antichain,
and similarly if $A_i\neq\0$ for some $i>n/2$.

When $n=2k+1$ the same argument shows that any largest
antichain of $\XX$ is $C\cup D$ with
$C\sub \gG_k$ and $D\sub \gG_{k+1}$.
But then the union of the closures, say $A$ and $B$, of $C$ and $D$
is an antichain
of $2^{[n]}$, so
$\min\{|A|,|B|\}\leq \frac{1}{2}\C{n}{k}$;
and if
(e.g.) this minimum is $|A|>0$, then, according to \eqref{X1},
replacing $C$ by
$\XX\cap \partial_u(A)$ ($=\XX\cap \partial_u(C)$) increases the size of our antichain,
a contradiction.\qed

\bn
Department of Mathematics\\
Rutgers University\\
Piscataway NJ 08854\\
hammac3@math.rutgers.edu\\
jkahn@math.rutgers.edu

\end{document}